\documentclass[12pt]{amsart}
\usepackage{primes}

\title[Representation Stability and Finite Orthogonal Groups]{Representation Stability and Finite Orthogonal groups}
\author[A.S. Kannan]{Arun S. Kannan}
\author[Zifan Wang]{Zifan Wang}
\address{Department of Mathematics, Massachusetts Institute of Technology, Cambridge, MA 02139}
\email{akannan@mit.edu}
\address{Princeton International School of Mathematics and Science, Princeton, NJ 08540}
\email{zifanawang04@gmail.com}
\begin{document}

\begin{abstract}
In this paper, we prove stability results about orthogonal groups over finite commutative rings where 2 is a unit. Inspired by Putman and Sam (2017), we construct a category $\OrI(R)$ and prove a Noetherianity theorem for the category of $\OrI(R)$-modules. This implies an asymptotic structure theorem for orthogonal groups. In addition, we show general homological stability theorems for orthogonal groups, with both untwisted and twisted coefficients, partially generalizing a result of Charney (1987).
\end{abstract}

\maketitle

\setcounter{tocdepth}{1}
\tableofcontents

\section{Introduction}

\subsection{A motivating example}


Consider a compact oriented manifold $X$ with nonempty boundary. The configuration space of $n$ points in $X$ is defined as the space of all possible choices of $n$ distinct points in $X$:
\[\mathrm{Conf}_n(X) = X^n \big\backslash \{(x_1,x_2,\dots,x_n) \mid x_i = x_j \text{ for some } i\neq j\}.\]
There is a natural action of the symmetric group $S_n$ on $\mathrm{Conf}_n(X)$ by permuting the coordinates. Note that there exist natural maps $\mathrm{Conf}_{n+1}(X) \to \mathrm{Conf}_{n}(X)$ by forgetting the $(n+1)$th coordinate. These maps induce morphisms at each level of cohomology, with coefficients in a fixed field $\KK$ (or more generally a Noetherian ring) with $\operatorname{char} \KK = 0$: \[H^i(\mathrm{Conf}_{n}(X); \KK) \to H^i(\mathrm{Conf}_{n+1}(X); \KK).\]

Denote by $\mathrm{UConf}_n(X)$ the unordered configuration space, which is the orbit space of the action of $S_n$ on $\mathrm{Conf}_n(X)$. In the unordered case, it is a classical result \cite{McDuff75} that for any fixed index $i$, $H^i(\mathrm{UConf}_{n}(X); \KK) \cong H^i(\mathrm{UConf}_{n+1}(X); \KK)$ for all $n\gg 0$. For $\mathrm{Conf}_n(X)$, however, this does not hold. Instead, we have \emph{representation stability}: notice that the action of $S_n$ on $\mathrm{Conf}_n(X)$ induces a representation of $S_n$ on $V_n := H^i(\mathrm{Conf}_{n}(X); \KK)$, and therefore $V_n$ splits into the direct sum of irreducible representations, which are naturally parametrized by partitions of $n$. For any partition $\lambda$ of a positive integer $k$, we write $c_{\lambda}(V_n)$ to denote the multiplicity inside the expansion of $V_n$ of the irreducible representation corresponding to the partition $(n-k, \lambda)$ of $n$. In \cite{Church12}, it was shown that there exists $N$ such that the following three properties hold for all $n>N$:
\begin{enumerate}
    \item Injectivity: the maps $V_n\to V_{n+1}$ are injective;
    \item Surjectivity: the image of $V_n$ spans $V_{n+1}$ as a $\KK[S_{n+1}]$-module;
    \item Multiplicity stability: for any $k$ and any partition $\lambda$ of $k$, $c_{\lambda}(V_n) = c_{\lambda}(V_{n+1})$.
\end{enumerate}

A direct corollary is that the dimensions of $V_n$ exhibit polynomial growth as $n\gg 0$.

\subsection{Background, history, and known results}

In general, the notion of representation stability considers a sequence of representations $V_n$ of a family of groups $G_n$, with maps $V_n\to V_{n+1}$ and $G_n\to G_{n+1}$ between them that are compatible with the action of $G_n$ on $V_n$. This framework was first introduced in \cite{Church_2013} to describe the frequent observation that various representation-theoretic properties of $V_n$ stabilize for sufficiently large $n$. Stability came to possess a very broad meaning: one example is homological stability, which holds for $\mathrm{UConf}_n(X)$ as discussed above. Another example is multiplicity stability, which holds for $\mathrm{Conf}_n(X)$. In $\mathsection \ref{asymp}$ of this paper, we prove for orthogonal groups a categorical version of representation stability, phrased in terms of Kan extensions.

A framework was developed in \cite{Church_2015}, involving the functor category of modules over a category named $\mathbf{FI}$. This category of $\mathbf{FI}$-modules carries the rich structure of an abelian category, and it acts as a large algebraic structure which altogether controls the growth behavior of a sequence of representations, thus providing a fundamental explanation of various stability behavior associated with representations of the symmetric groups $S_n$.

The objects of the category $\mathbf{FI}$ are \underline{f}inite sets, and the morphisms between them are \underline{i}njections (hence the name FI). The reason that $\mathbf{FI}$ plays a prominent role in the representation stability of symmetric groups is because the automorphism group of an $n$-element object in $\mathbf{FI}$ is precisely $S_n$. Key to the utility of $\mathbf{FI}$-modules is a certain Noetherian property that roughly states that any submodule of a finitely generated $\mathbf{FI}$-module is also finitely generated. This local Noetherianity property is a key ingredient in \cite{Church_2015}'s proof of representation stability of $\mathrm{Conf}_n(X)$.

In \cite{Putman_2017}, analogues ($\mathbf{VIC}$-modules, $\mathbf{SI}$-modules) to the functor category of $\mathbf{FI}$-modules were constructed by replacing the symmetric groups with the general linear groups and the symplectic groups (over finite rings). Similar Noetherian properties and asymptotic structure theorems were proven, as well as broad homological stability theorems with twisted coefficients. Some of these results were strengthened in \cite{miller2020quantitative}, where an explicit bound for the stability degree was shown.

\subsection{Main results}

Our motivation for this paper was that among the classical groups, only the orthogonal groups were not studied much in the context of representation stability. The main obstacle in this case stems from the fact that, unlike alternating forms and symplectic groups, more than one symmetric bilinear forms exist (up to isometry), even for modules over finite rings.

In this paper, we address this problem for orthogonal groups over finite commutative rings $R$ where 2 is a unit. Extending methods in \cite{Putman_2017}, we construct a category $\OrI(R)$ and the functor category of $\OrI(R)$-modules, and show a Noetherian property (Theorem \ref{OrI_noetherian}) analogous to that of $\mathbf{FI}$-modules. We then prove that this implies the analogous version of the asymptotic structure theorem for orthogonal groups (Theorems \ref{surj}, \ref{central} and Corollary \ref{inj}).

In addition, applying our Theorem \ref{OrI_noetherian}, we showed general homological stability theorems (Theorem \ref{twisted}) with twisted coefficients for (indefinite) orthogonal groups. Interestingly, these results yield reverse implications for the usual stability with untwisted coefficients, partially generalizing a result in \cite{charney1987}. It is possible that our results also have interesting implications in the context of mapping class groups of high-dimensional manifolds.

\subsection{Roadmap}

This paper is structured as follows. In $\mathsection \ref{prelim}$, we review the framework of functor categories of modules over a category, recall the definition of column-adapted maps, and recount the theory of symmetric bilinear forms over finite commutative rings. In $\mathsection \ref{noether}$, we define the category $\OrI(R)$ and the category of $\OrI(R)$-modules, and show that the latter is locally Noetherian. In $\mathsection \ref{asymp}$, we show the asymptotic structure theorem, which is a strong stability result for finitely generated $\OrI(R)$-modules. Finally, in $\mathsection \ref{twist}$, we show homological stability with twisted coefficients (coefficients determined by some finitely generated $\OrI(R)$-module), as well as its interesting reverse implication for stability with untwisted coefficients.

\subsection*{Acknowledgments}

This paper is the result of PRIMES-USA, a program in which high school students (the second author) engage in research-level mathematics led by a mentor (the first author). This paper is based upon work supported by The National Science Foundation Graduate Research Fellowship Program under Grant No.~1842490 awarded to the first author. The authors would like to thank the organizers of the PRIMES-USA program for providing the opportunity for this research. We are also grateful to Dr. Tanya Khovanova and Dr. Kent Vashaw for their helpful comments and to Prof. Steven V. Sam for suggesting this research topic and providing valuable suggestions. 

\section{Preliminaries}

\label{prelim}

\subsection{Finitely generated $\mathcal{C}$-modules}

We begin by reviewing the necessary framework of functor categories of modules over a category.

\begin{Def}
  Let $\Cl$ be a category and $\KK$ a ring. A \emph{$\Cl$-module over $\KK$} is a functor $M: \Cl \rightarrow \mathbf{Mod}_\KK$, where $\mathbf{Mod}_\KK$ is the category of $\KK$-modules. If the ring $\KK$ is clear from the context, we shall just use the term \emph{$\Cl$-module}. A \emph{$\Cl$-module homomorphism} $\eta: M \rightarrow N$ between two $\Cl$-modules $M$ and $N$ is a natural transformation of functors. The category of all $\Cl$-modules forms a category, which we call $\Rep_{\KK}(\Cl)$.
\end{Def}

A $\Cl$-module homomorphism $\eta: M\to N$ is \emph{injective} (resp. \emph{surjective}) if for each object $C\in \Cl$, the component $\eta_C: M(C)\to N(C)$ is injective (resp. surjective). We say $N$ is a \emph{submodule} (resp. \emph{quotient module}) of $M$ if there is an injective (resp. surjective) homomorphism $N \to M$ (resp. $M\to N$). It is a well-known fact that concepts such as subobjects, quotients, kernels, cokernels, images, direct sums, etc. can all be defined in this ``pointwise'' fashion in the context of $\Cl$-modules. In other words, $\Rep_{\KK}(\Cl)$ has the structure of an abelian category.
 
One of the key ingredients in \cite{Church_2015} is the notion of a Noetherianity property. Recall that a module over a ring $\KK$ is Noetherian if every submodule is finitely generated (assuming the axiom of choice). The following definitions generalize these notions to $\Cl$-modules.

\begin{Def} \label{def_fin_gen}
A $\Cl$-module $M$ is \emph{finitely generated} if there exist objects $C_1, C_2, \dots C_n \in \Cl$ and elements $x_i \in M(C_i)$ for each $i$, satisfying that if $N$ is a submodule of $M$ such that $N(C_i)$ contains $x_i$, then $N = M$. The set $\{x_i\}$ is called the \emph{generating set} of $M$.
\end{Def}

\begin{Def}
A $\Cl$-module $M$ is \emph{Noetherian} if every submodule of $M$ is finitely generated. The category of $\Cl$-modules is \emph{locally Noetherian} if for any Noetherian ring $\KK$, all finitely generated $\Cl$-modules are Noetherian.
\end{Def}

An equivalent formulation of Definition \ref{def_fin_gen} is occassionally useful. For any object $X \in \Cl$, let $P_{\Cl, X}$ denote the covariant representable $\Cl$-module generated at $X$, i.e. the functor defined by
\begin{align*}
P_{\Cl, X}: \Cl &\to \textbf{Mod}_{\KK} \\
Y &\to \KK[\Hom_{\Cl}(X,Y)]
\end{align*}
for all $Y \in \Cl$. Then the following lemma holds:

\begin{lem}
\label{fingen_quotient}
A $\Cl$-module is finitely generated if and only if it is a quotient of a direct sum of modules of the form $P_{\Cl, X}$.
\end{lem}

\begin{proof}
By the Yoneda lemma, a $\Cl$-module homomorphism $\eta: P_{\Cl, X} \rightarrow M$ is determined uniquely by choosing an element $x \in M(X)$ and letting $\eta_X(1_X) = x$. It is straightforward to check that if $M$ is finitely generated, then $M$ is a quotient of the direct sum of the representable functors attached to the generating set. Similarly, if $M$ is a quotient of this form, then the elements corresponding to each representable functor will be a generating set of $M$.
\end{proof}

Let $f:\Cl \to \mathcal{D}$ be a functor. This induces a functor $f^*: \Rep_{\KK}(\mathcal{D})\to\Rep_{\KK}(\mathcal{C})$. The functor $f$ is defined to be \emph{finite} if for every $X\in\mathcal{D}$, $f^*(P_{\mathcal{D},X})$ is finitely generated. We end this subsection by recalling the following two lemmas, which appeared respectively as Lemmas 2.1 and 2.2 in \cite{Putman_2017}.

\begin{lem} 
\label{Noetherian_condition}
Let $\Cl$ be a category. The category of $\Cl$-modules is locally Noetherian if and only if for any object $X\in\Cl$, any submodule of $P_{\Cl,X}$ is finitely generated.
\end{lem}

\begin{lem} 
\label{Noetherian_transfer}
If the category of $\Cl$-modules is locally Noetherian, and $f: \Cl\to\mathcal{D}$ is a finite and essentially surjective functor, then the category of $\mathcal{D}$-modules is locally Noetherian.
\end{lem}

\subsection{Semilocal rings and finite rings}

In this paper, we shall consider only finite commutative rings with unit. However, the literature of orthogonal forms often deal with semilocal rings, which are more general, so we briefly mention them here. Recall that a ring is \emph{Artinian} if there is no infinite descending chain of ideals, and a ring $R$ is \emph{semilocal} if $R/\operatorname{rad} R$ is Artinian. We recall an equivalent characterization of semilocal rings (c.f. \cite{lam2001first}):

\begin{prop}
A ring $R$ is semilocal if and only if it has finitely many maximal ideals.
\end{prop}

Therefore, it is clear that a finite ring is semilocal. Furthermore, if $\mathfrak{p}$ is a prime ideal in a finite ring $R$, then $R/\mathfrak{p}$ is a finite integral domain and therefore a field. This implies that $\mathfrak{p}$ is maximal. Therefore, we have the following result (c.f. \cite{lam2001first}):

\begin{prop}\label{decomp}
A finite ring $R$ is the direct product of finite local rings.
\end{prop}

Therefore, given a finite commutative ring $R$, we can express it as the product of finite local rings $R = \prod_{i = 1}^n R_i$. Then, since each $R_i$ is local, there is a unique maximal ideal $\mathfrak{m}_i$ in each $R_i$, and this gives a projection map $\pi_i: R_i \rightarrow R_i/\mathfrak{m}_i$, the codomain of which is a field; in particular, the product map $\pi = \prod_{i=1}^{n} \pi_i$ gives a projection map from $R$ to a product of finite fields $R/\mathfrak{m}$, where $\mathfrak{m} = \prod_{i=1}^n \mathfrak{m}_i$.

\subsection{Symmetric bilinear forms}
We now wish to define and characterize symmetric bilinear forms on finite commutative rings. As revealed shortly, we need to assume that 2 is a unit.

\begin{Def}
Let $R$ be a semilocal ring, and let $V$ be a finite-rank free $R$-module. A bilinear form $B: V \times V \rightarrow R$ is called \emph{symmetric} or \emph{orthogonal} if $B(v,w) = B(w,v)$ for all $v, w$. The form is said to be \emph{non-degenerate} if it it induces an isomorphism to the dual space $V^* = \Hom_R(V, R)$. If $B$ is non-degenerate, call the pair $(V, B)$ an \emph{orthogonal module}. If $(V, B_V)$ and $(W, B_W)$ are two orthogonal modules, an $R$-module homomorphism $\phi: V \rightarrow W$ is called an \emph{isometry} if $B_V(v, w) = B_W(\phi(v), \phi(w))$ for all $v, w \in V$.
\end{Def}

From the definition, it follows that isometries are necessarily injective.

The classification of orthogonal modules over a finite ring up to bijective isometry is more difficult than the symplectic case. We first recall the following diagonalization theorem (c.f. \cite{baeza2006quadratic}):

\begin{prop}\label{bilformdiag}
Let $R$ be a semilocal ring, and let $V$ be an orthogonal $R$-module. Then there exists a basis of $V$ in which the matrix of $B$ is diagonal and whose diagonal entries are units in $R$.
\end{prop} 

In other words, we can find a bijective isometry from $(V, B)$ to $(R^{\rk V}, D)$, where $D$ is a diagonal form as in the theorem. (Here, $\rk V$ denotes the rank of a free $R$-module $V$.) While this theorem greatly simplifies the classification problem, it it still redundant (for instance, permuting basis vectors in $R^{\rk V}$ will change $D$ but the resulting module is still isometric). In the case where $R$ is a finite field, the answer is well-known (though a proof is hard to find, c.f. \cite{glasser05}):

\begin{prop}\label{bilformfields}
Let $\mathbb{F}$ be a finite field (of characteristic $p > 2$), and let $(V,B)$ be an orthogonal $\mathbb{F}$-module (i.e. a finite-dimensional vector space endowed with a non-degenerate symmetric bilinear form). Then, there exists a basis of $V$ such that matrix of $B$ is either 1) the identity matrix, or 2) the diagonal matrix $diag(1, \dots, 1, x)$, where $x$ is any nonsquare in $\mathbb{F}^\times$, where different choices of $x$ yield isometric forms.
\end{prop}
In other words, there are two isomorphism classes, and the dimension of $V$ and the determinant of $B$ determine the isomorphism class.

A similar result can be proven for a finite local ring. If $R$ is a finite local ring, let $\pi: R \rightarrow \mathbb{F}$ denote the projection onto its residue field $\mathbb{F} = R/\mathfrak{m}$, where $\mathfrak{m}$ is the maximal ideal in $R$. Then we have the following well-known proposition (we provide a proof because we could not find a proof in the literature):

\begin{prop}\label{bilformlocalrings}
Let $R$ be a finite local ring (where $2$ is a unit), and let $(V,B)$ be an orthogonal $R$-module. Then, there exists a basis of $V$ such that matrix of $B$ is either 1) the identity matrix, or 2) the diagonal matrix $diag(1, \dots, 1, x)$, where $x \in R$ is such that $\pi(x)$ is a nonsquare in $\mathbb{F}^\times$, and where different choices of $x$ yield isometric forms.
\end{prop}

\begin{proof}
First of all, since $R$ is a local ring, $\mathfrak{m}$ consists of the non-units in $R$, so for any unit $u \in R$, the coset $u + \mathfrak{m}$ consists solely of units. By Proposition \ref{bilformdiag} and applying Proposition \ref{bilformfields} to the induced orthogonal $\mathbb{F}$-module, we can find a basis $\{v_1, \dots, v_n\}$ of $V$ such that with respect to this basis the form is diagonal, $B(v_i, v_i) = (1 + t_i)^{-1}$ where $t_i \in \mathfrak{m}$ for each $1 \leq i \leq n - 1$ and $B$ satisfies one of the following two cases: either $B(v_n, v_n) = (1 + t_n)^{-1}$ or $B(v_n, v_n) = (x + t_n)^{-1}$, where $t_n \in \mathfrak{m}$ and $x$ is a unit in $R$ such that $\pi(x)$ is a nonsquare in $\mathbb{F}^{\times}$. 

Let us do case 1) first. For each $i$ in $\{1, \dots, n\}$, consider the following quadratic equation in $m$: $(1+m)^2 = 1 + t_i$, which can be rewritten as $m^2 + 2m - t_i = 0$. Since $t_i \in \mathfrak{m}$, reducing this monic polynomial modulo $\mathfrak{m}$ gives a monic quadratic equation with two distinct roots $m(m + 2) = 0$, one of them being $m = 0$. By Theorem 3.12 in \cite{ganske1973finite}, it follows $m^2 + 2m - t_i = 0$ has a root $m_i$ in $\mathfrak{m}$. Then, in the basis $\{(1 + m_1)v_1, \dots, (1 + m_n)v_n\}$, we have $B((1+m_i)v_i, (1+m_i)v_i) = (1+m_i)^2B(v_i, v_i) = (1+m_i)^2(1+t_i)^{-1} = 1$ for all $i$.

For case 2), we can do the same thing for $1 \leq i \leq n-1$. For $i = n$, we consider the polynomial equation $(x+m)^2 = x(x + t_n)$, which when reduced modulo $\mathfrak{m}$ gives $m(m+2\pi(x)) = 0$. The same reasoning gives a root $m = m_n \in \mathfrak{m}$ of  $(x+m)^2 = x(x + t_n)$. Then, we have $B((x + m_n)v_n, (x+m_n)v_n) = (x+m_n)^2(x+t_n)^{-1} = x$. This proves the theorem.
\end{proof}

\begin{cor}
\label{finite_ring_decomp}
Let $R$ be a finite ring (where $2$ is a unit), and write $R = \prod_{i=1}^q R_i$ as the product of finite local rings. Then, there are $2^q$ isomorphism classes of orthogonal $R$-modules. 
\end{cor}

\begin{proof}
Such a decomposition exists by Proposition \ref{decomp}. Let $e_i = (0, \dots, 0, 1_{R_i}, 0, \dots, 0)$ (nonzero in the $i$-th spot) be the central idempotent arising from $R_i = e_iR$, so $R = \bigoplus_{i=1}^{n} e_iR$. Then, since $1_R = e_1 + \cdots + e_n$, it is clear that a bilinear form on $R$ splits as the direct sum of bilinear forms on $R_i$. The result then follows from Theorem \ref{bilformlocalrings}.
\end{proof}

\subsection{Column-adapted maps}

In this last subsection, we discuss the notion of column-adapted and row-adapted maps, introduced in \cite{Putman_2017}. They are used in our proof of Lemma \ref{OrI_factor_lem}.

\begin{Def}
Let $R$ be a commutative local ring. An $R$-linear map $f:R^m\to R^n$ is \emph{column-adapted} if there is an $n$-element subset $S_c(f) = \{s_1<s_2<\dots<s_n\}\subseteq [m]$ such that, if we write $f$ as a $n\times m$ matrix $M$ with respect to the standard basis, then
\begin{itemize}
\item The $s_i$th column of $M$ has 1 on the $i$th position and 0 elsewhere;
\item The entries $(i,j)$ where $j<s_i$ are all non-invertible.
\end{itemize}
\end{Def}

For example, the map $f:R^5\to R^3$ defined by the matrix
\[
\left(\begin{matrix}
\ast & 1 & 0 & \bullet & 0 \\
\ast & 0 & 1 & \bullet & 0 \\
\ast & 0 & 0 & \ast & 1
\end{matrix}\right)
\]
is column adapted, if the entries labeled with $\ast$ are non-invertible (the entries labeled with $\bullet$ can be any scalar).

In the general case where $R$ is a finite commutative ring, by Proposition \ref{decomp} there exists an isomorphism
\[R \cong R_1\times \dots \times R_q\]
where the $R_i$'s are finite commutative local rings. In this case, we say a map $f:R^m\to R^n$ is \emph{column-adapted} if the induced maps $R_i^m \to R_i^n$ are all column-adapted. Also, we say $f$ is \emph{row-adapted} if its transpose is column-adapted; in this case, we also define $S_r(f) = S_c(f^{T})$.

The next two lemmas were established in \cite{Putman_2017} as Lemmas 2.9 and 2.10.

\begin{lem}
The composition of two column-adapted maps is column-adapted. Similarly, the composition of two row-adapted maps is row-adapted.
\end{lem}

\begin{lem} \label{factor_lem}
Let $R$ be a finite commutative ring, and let $f:R^{n'} \to R^{n}$ be a surjection. Then we can uniquely factor $f=f_2f_1$, where $f_1: R^{n'} \to R^n$ is column-adapted and $f_2: R^n\to R^n$ is an isomorphism.
\end{lem}

%

\section{Local Noetherianity}

\label{noether}

Throughout this section, we fix a finite commutative ring $R$ where 2 is a unit. Suppose we fix a factorization $R \cong R_1\times R_2 \times \dots \times R_q$, where each $R_i$ is a local ring.

\begin{Def}
Define the following categories:
\begin{itemize}
\item $\OrI(R)$: objects are \underline{or}thogonal $R$-modules $(V,B)$, and morphisms are \underline{i}sometries.
\item $\OrIsq(R)$: the full subcategory of $\OrI(R)$ spanned by objects $(V,B)$ such that the projection of $\det B$ onto the residue fields of each $R_i$ ($1\le i \le q$) is a square.
\end{itemize}

\end{Def}

We point out the important point that for any object $(V,B)\in \OrI(R)$, its group of automorphisms $\Aut_{\OrI(R)}(V)$ is precisely the orthogonal group associated to $(V,B)$. 

The goal in this section is to prove the following theorem:

\begin{thm} \label{OrI_noetherian}
For a finite commutative ring $R$ and a Noetherian ring $\KK$, the category \[\Rep_{\KK}(\OrI(R))\] is locally Noetherian.
\end{thm}

Our plan is to show the theorem with $\OrI(R)$ replaced with $\OrIsq(R)$ in $\mathsection$3.1, and generalize it to the full category in $\mathsection$3.2.

\subsection{The square case}
In this subsection, our goal is to show that the category of $\OrIsq(R)$-modules is locally Noetherian (stated as Theorem \ref{OrIsq_noetherian} below). The proof is essentially identical to the analogous argument in \cite{Putman_2017}.

\begin{Def}
Define the following categories:
\begin{itemize}
\item $\OOrIsq'(R)$: objects are orthogonal $R$-modules $(R^n,B)$ such that the projection from $\det B$ onto $R_i$ is a square, and morphisms are row-adapted isometries.
\item $\OOrIsq(R)$: full subcategory of $\OOrIsq'(R)$ spanned by objects which are orthogonal $R$-modules $(R^n,B_{\sq})$ such that $B_{\sq}$ is the identity matrix under the standard basis.
\end{itemize}
\end{Def}

\begin{lem} \label{OrI_factor_lem}Let $R$ be a finite commutative ring where 2 is a unit. Let \[f\in \Hom_{\OrIsq(R)}((R^{n},B),(R^{n'},B')),\] then we can uniquely write $f=f_1f_2$ such that 
\begin{itemize}
\item $f_2\in\Hom_{\OrIsq(R)}((R^{n}, B), (R^{n}, \beta))$ where $\beta$ has square determinant;
\item $f_1 \in \Hom_{\OOrIsq'(R)}((R^{n},\beta) \to (R^{n'},B'))$.
\end{itemize}
\end{lem}

\begin{proof}
Transposing, applying Lemma \ref{factor_lem}, then transposing back, we can uniquely write $f=f_1f_2$, where $f_1 : R^{n}\to R^{n'}$ is row-adapted and $f_2:R^{n}\to R^{n}$ is an isomorphism. We can then uniquely choose $\beta$ a symmetric form on $R^{n}$ so that $f_1$ and $f_2$ are isometries. 
\end{proof}

Our first step in achieving the goal of this subsection is to show the existence of a well partial ordering $\preceq$ on the set
\[\mathcal{P}_R(d,B) = \bigsqcup_{n\ge 0} \Hom_{\OOrIsq'(R)} ((R^{d},B), (R^{n},B_{\sq}))\]
as described in the following lemma.

\begin{lem} \label{OrI_order}
Fix $R, d, B$. There exists a well partial ordering $\preceq$ on $\mathcal{P}_R(d,B)$ that can be extended to a total ordering $\le$ such that for $f,g\in \mathcal{P}_R(d,B)$, mapping to $R^{n}, R^{n'}$ respectively, satisfying $f\preceq g$, there exists some $\phi \in \Hom_{\OOrIsq(R)}((R^{n}, B_{\sq}), (R^{n'},B_{\sq}))$ such that:
\begin{itemize}
    \item $g=\phi f$;
    \item For any $f_1 \in \Hom_{\OOrIsq'(R)}((R^{d}, B), (R^{n},B_{\sq}))$ with $f_1<f$, $\phi f_1 < g$.
\end{itemize}
\end{lem}

To prove Lemma \ref{OrI_order}, we need to first show the following technical lemma. The proof of Lemma \ref{OrI_order} is located shortly after it.

\begin{lem} \label{OrI_insertion}
Fix $R,d,B,f,g,n,n'$ as described in Lemma \ref{OrI_order}, except here we further restrict to the case where $R$ is a finite commutative local ring. Let the rows of $g$ be $r_1,\dots,r_{n'}$. Suppose that $f$ can be obtained from $g$ by deleting certain rows $r_{i_1}, \dots, r_{i_{n'-n}}$ where $I=\{i_1,\dots,i_{n'-n}\}\subseteq [n']$ such that $I \cap S_r(g) = \varnothing$. Then there exists $\phi \in \Hom_{\OOrIsq(R)}(R^{n}, R^{n'})$ such that:
\begin{itemize}
\item For any $h \in \Hom_{\OOrIsq'(R)}((R^{d},B), (R^{n}, B_{\sq}))$ with $S_r(h) = S_r(f)$, $\phi h$ can be obtained from $h$ by inserting the rows $r_i$ in position $i$ for each $i\in I$. In particular, $g = \phi f$.
\item For any $h \in \Hom_{\OOrIsq'(R)}((R^{d},B), (R^{n}, B_{\sq}))$ with $S_r(h) < S_r(f)$ in lexicographic order, then $S_r(\phi h) < S_r(g)$ in lexicographic order.
\end{itemize}
\end{lem}

\begin{proof}
The desired map $\phi$ can be defined by the following $n'\times n$ matrix: take a $(n'-n)\times n$ matrix whose $k$th row $\widehat{r_k}$ is obtained from $r_{i_k}$ by replacing the entries not in $S_r(f)$ with zeros. Then, we shuffle the rows of this matrix with the rows of a $n\times n$ identity matrix, such that the former rows occupy exactly the indices in $I$. It is straightforward to check that the two points hold and $\phi$ preserves the standard symmetric forms.

As a toy example, let $(d,n,n')=(3,4,6)$, and let
\[
f = \left[\begin{matrix}
1 & 0 & 0 \\
a & b & c \\
0 & 1 & 0 \\
0 & 0 & 1
\end{matrix}\right] \text{ and }
g = \left[\begin{matrix}
1 & 0 & 0 \\
a & b & c \\
0 & 1 & 0 \\
d & e & f \\
0 & 0 & 1 \\ 
g & h & i
\end{matrix}\right].
\]
Then we take $\phi$ to be
\[
\phi = \left[\begin{matrix}
1 & 0 & 0 & 0 \\
0 & 1 & 0 & 0 \\
0 & 0 & 1 & 0 \\
d & 0 & e & f \\
0 & 0 & 0 & 1 \\
g & 0 & h & i
\end{matrix}\right].
\]
Consider two vectors $v_i=(w_i,x_i,y_i,z_i)$ where $i\in\{1,2\}$, let us now prove $B_{\sq}(v_1,v_2) = B_{\sq}(\phi(v_1),\phi(v_2))$. We know, by definition of $f$ and $g$, that 
\[B_{\sq}(f(w_1,y_1,z_1),f(w_2,y_2,z_2)) = B((w_1,y_1,z_1),(w_2,y_2,z_2)) = B_{\sq}(g(w_1,y_1,z_1),g(w_2,y_2,z_2)),\]
so
\[(dw_1+ey_1+fz_1)(dw_2+ey_2+fz_2) + (gw_1+hy_1+iz_1)(gw_2+hy_2+iz_2) = 0,\]
which is equivalent to $B_{\sq}(v_1,v_2) = B_{\sq}(\phi(v_1),\phi(v_2))$. In the general case,
let the rows of $f$ be $f_i$ ($1\le i\le n$), and the rows of $g$ be $g_i$ ($1\le i\le n'$). Let $\alpha = (\alpha_i),\beta = (\beta_i)$ be any two vectors in $R^n$, we want to prove that $B_{\sq}(\alpha,\beta) = B_{\sq}(\phi(\alpha),\phi(\beta))$. Let $\alpha',\beta'$ denote the vectors in $R^d$ obtained from $\alpha,\beta$ by only selecting the entries with indices in $S_r(f)$. Then
\begin{align*}
B_{\sq}(\alpha,\beta) &= \sum_{i=1}^n \alpha_i\beta_i; \\
B_{\sq}(\phi(\alpha),\phi(\beta)) &= \sum_{i=1}^n \alpha_i\beta_i + \sum_{i\in I}(g_i\cdot \alpha')(g_i\cdot \beta'); \\
B_{\sq}(f(\alpha'),f(\beta')) &= \sum_{i=1}^n (f_i\cdot \alpha')(f_i\cdot \beta'); \\
B_{\sq}(g(\alpha'),g(\beta')) &= \sum_{i=1}^{n'} (g_i\cdot \alpha')(g_i\cdot \beta').
\end{align*}
Because $B_{\sq}(f(\alpha'),f(\beta')) = B(\alpha,\beta) = B_{\sq}(g(\alpha'),g(\beta'))$, we conclude that 
\[B_{\sq}(\alpha,\beta) = B_{\sq}(\phi(\alpha),\phi(\beta)).
\]
\end{proof}

We now give the proof of Lemma \ref{OrI_order}.

\begin{proof}[Proof of Lemma \ref{OrI_order}]
Assume first that $R$ is a finite commutative local ring.
Let $f,g\in\mathcal{P}_R(d,B)$ mapping to $R^n, R^{n'}$ respectively. We declare $f\preceq g$ if $f$ can be obtained from $g$ by deleting some set of rows $I\subseteq[n']$ such that $I\cap S_r(g) = \varnothing$. This is clearly a partial order.

Let $\Sigma = R^d \sqcup \{\bullet\}$, $\bullet$ being a formal symbol. Let $\Sigma^*$ be the set of words whose letters come from $\Sigma$. There is a natural well-ordered poset structure on $\Sigma^*$ (cf. \cite{Putman_2017} Lemma 2.5). We will prove that $(\mathcal{P}_R(d,B),\preceq)$ is isomorphic to a subposet of $\Sigma^{\ast}$, which would imply that $\preceq$ is a well partial ordering. For $f\in\Hom_{\OOrIsq'(R)}((R^{d},B),(R^{n},B_{\sq}))$, let $r_i$ represent the $i$th row of $f$ if $i\notin S_r(f)$, else let $r_i = \bullet$. Thus, each $r_i \in \Sigma$, and we map $f$ to the word $r_1 r_2 \dots r_n \in \Sigma^{\ast}$. Clearly this map is an order-preserving injection, so we conclude that $\preceq$ is a well partial ordering.

Next, we extend $\preceq$ to a total ordering $\leq$. Fix an arbitrary total order on $R^{d}$. For $f\neq g$ the order is defined by
\begin{itemize}
    \item If $n<n'$ then $f<g$;
    \item Otherwise, if $S_r(f)<S_r(g)$ in lexicographic order, then $f<g$;
    \item Otherwise, compare the sequences of rows of $f$ and $g$ by lexicographic order and the total order on $R^{d}$.
\end{itemize}
This clearly extends $\preceq$, and the claimed properties follow by taking $\phi$ as described by Lemma \ref{OrI_insertion}. This shows Lemma 3.5 in the case where $R$ is a finite commutative local ring.

In the general case where $R$ is a finite commutative ring, fix an isomorphism
\[R \cong R_1\times R_2\times \dots \times R_q\]
where each $R_i$ is a commutative local ring. Then
\[\Hom_{\OOrIsq'(R)}(R^n, R^{n'}) = \Hom_{\OOrIsq'(R_1)}(R_1^n, R_1^{n'}) \times \dots \times \Hom_{\OOrIsq'(R_q)}(R_q^n, R_q^{n'}).\]
This implies that each element $f \in \mathcal{P}_R(d, B)$ can be viewed as a tuple \[(f_1,\dots,f_q) \in \mathcal{P}_{R_1}(d, B) \times \dots \times \mathcal{P}_{R_q}(d, B),\]
where each $f_i$ maps into $R_i^{n'}$ for fixed $n'$.
We can construct a partial order on the product by using the product partial order (where one element is smaller than another iff every component is). Then, extend this to a total order by lexicographic order. This restricts to a total order on $\mathcal{P}_R(d, B)$ that satisfies the necessary assumptions of Lemma \ref{OrI_order}.
\end{proof}

Using this well ordering, we can deduce the following theorem. The proof closely follows ideas in Section 4 of \cite{SamSnowden_2017}.

\begin{thm} \label{OOrI_noetherian}
Let $R$ be a finite commutative ring. For $d\ge 0$ and $B$ a symmetric form on $R^{d}$, any $\OOrIsq(R)$-submodule of the $\OOrIsq(R)$-module
\[Q_{d,B} = \KK[\Hom_{\OOrIsq'(R)}((R^{d},B),-)]\]
is finitely generated. As a corollary, the category of $\OOrIsq(R)$-modules is locally Noetherian.
\end{thm}

\begin{proof}
In view of Lemma \ref{Noetherian_condition}, it suffices to prove that any submodule of $Q_{d,B}$ is finitely generated (since we can take $B=B_{\sq}$).

Fix $d,B,R,\KK$, so we abbreviate $Q_{d,B}$ as $Q$. For an element $f\in \Hom_{\OOrIsq'(R)}((R^{d},B),R^{n})$, let $e_f$ denote the basis vector in $Q(R^{n})$ corresponding to $f$. For an element $x\in Q(R^{n})$, define its \textbf{initial term} $\init(x)$ as follows: if $f$ is $\le$-maximal such that $e_f$ has coefficient $\alpha_f \neq 0$ in $x$, $\init(x) = \alpha_f e_f$. Let $M$ be a submodule of $Q$, we also define $\init(M)$ to be a function taking $R^{n}$ to the $\KK$-module $\KK[\init(x) \mid x\in M(R^{n})]$.

We claim that if $N$ is a submodule of $M$ and $N\neq M$, then $\init(N) \neq \init(M)$. Suppose for contradiction that $\init(N) = \init(M)$. Pick $y\in M(R^{n})\backslash N(R^{n})$ such that $\init(y) = \alpha_t e_t$ is $\le$-minimal. Since $\init(M) = \init(N)$, there exists $z\in N(R^{n})$ such that $\init(z) = \init(y)$, but then $z-y\notin N(R^{n})$ and $\init(z-y)$ is smaller than $e_t$, contradiction. This proves the claim.

Suppose now that there exists a increasing sequence of submodules of $Q$
\[M_0 \subsetneq M_1 \subsetneq M_2 \subsetneq \dots.\]
The claim implies that $\init(M_i-1) \neq \init(M_{i})$, so there exists, for every $i\ge 1$, some $n_i\ge 0$ and $\lambda_i e_{f_i}\in \init(M_i)(R^{n_i})\backslash \init(M_{i-1})(R^{n_i})$. Because $\preceq$ is a well partial ordering, there exists an infinite sequence $i_0<i_1<i_2<\dots$ such that
\[f_{i_0} \preceq f_{i_1} \preceq f_{i_2} \preceq \dots\]
Since $\KK$ is Noetherian, we can choose $m$ such that $\lambda_{i_m} = \sum_{j=0}^{m-1} c_j\lambda_{i_j}$ for $c_j\in\KK$. For each $0\le j\le m-1$, let $x_j\in M_{i_j}(R^{n_{i_j}})$ such that $\init(x_j) = \lambda_{i_j}e_{f_{i_j}}$. By Lemma \ref{OrI_order}, there exists $\phi_j \in \Hom_{\OOrIsq}(R^{i_j}, R^{i_m})$ such that $\phi_j f_{i_j} = f_{i_{m}}$ and for any $f_{i_j}' < f_{i_j}$ in the same Hom set, $\phi_j f_{i_j} < f_{i_{m}}$.

Consider the element $X = \sum_{j=0}^{m-1} c_j \phi_j x_j$, which belongs to $M_{i_{m}-1}(R^{i_m})$. Then the properties in Lemma \ref{OrI_order} implies that $\init(X) = \lambda_{i_m} e_{f_{i_m}} \notin M_{i_{m}-1}(R^{i_m})$, contradiction.
\end{proof}

Finally, we are ready to show that the category of $\OrIsq(R)$-modules is locally Noetherian.

\begin{thm}
\label{OrIsq_noetherian}
Let $R$ be a finite commutative ring. The category of $\OrIsq(R)$-modules is locally Noetherian.
\end{thm}

\begin{proof}
By Lemma \ref{Noetherian_transfer} and Theorem \ref{OOrI_noetherian}, it suffices to show that the inclusion functor $\Phi: \OOrIsq(R) \to \OrIsq(R)$ is finite (the essential surjectivity of $\Phi$ is obvious). Fix $d,B$, and let $M=P_{\OrIsq(R), (R^{d},B)}$, it suffices to prove that the $\OOrIsq(R)$-module $\Phi^* = M\circ\Phi$ is finitely generated. Recall that by Theorem \ref{OOrI_noetherian}, $Q_{d,B'}$ is finitely generated for any symmetric form $B'$ with square determinant on $R^d$. If we fix $B'$ and an isometry $\tau: (R^{d},B) \to (R^{d},B')$, then we get a natural transformation $Q_{d,B'}\to \Phi^*$, and the map
\[\bigoplus_{B'} \bigoplus_{\tau} Q_{d,B'}  \to \Phi^*(M)\]
is surjective by Lemma \ref{OrI_factor_lem}. It follows then that $\Phi^*(M) = M\circ \Phi$ is finitely generated, as desired.
\end{proof}

\subsection{The general case}

In the general case, we define the following subcategories of $\OrI(R)$: for each nonempty subset $I$ of $\{1,2,\dots,q\}$, let $\OrI_I(R)$ be the full subcategory of $\OrI(R)$ spanned by the objects $(V,B)$, where $\det B$ is a nonsquare in the residue field of $R_i$ for all $i\in I$, and is a square otherwise. By Corollary \ref{finite_ring_decomp}, there are $2^q - 1$ of these categories. For each of them, we will show:

\begin{thm} \label{OrInsq_noetherian}
Let $R$ be a finite commutative ring. The category of $\OrI_I(R)$-modules is locally Noetherian.
\end{thm}

\begin{proof}
Define a functor $\Phi: \OrIsq(R) \to \OrI_I(R)$, sending
\[(V,B) \mapsto (V, B)\oplus(R,X)\]
where $X$ is the form on a 1-dimensional $R$-module sending $1$ to $x = (x_1,x_2,\dots,x_q) \in R$, where $x_i$ is an arbitrary nonsquare in the residue field of $R_i$ if $i\in I$, and is 1 otherwise. This functor sends every isometry $f: (V,B)\to(W,B')$ in $\OrIsq(R)$ to the isometry $\Phi(f): (V,B)\oplus(R,X) \to (W,B')\oplus(R,X)$ where the last component is preserved. By Lemma \ref{Noetherian_transfer}, because $\Phi$ is clearly essentially surjective, it suffices to check that it is finite, i.e. for any fixed $(V,B)\in\OrI_I(R)$, the $\OrIsq(R)$-module $P$ mapping
\[(W,B') \mapsto \KK[\Hom_{\OrI_I(R)}((V,B),(W,B')\oplus(R,X))]\]
is finitely generated.

Consider the representable $\OrI_I(R)$-module $P_{\OrI_I(R), (V,B)}$. This is clearly finitely generated. Pick a set of generators $x_1,\dots,x_n$. Then the same elements also generate $P$, which implies that it is finitely generated.
\end{proof}

Using Theorems \ref{OrIsq_noetherian} and \ref{OrInsq_noetherian}, we now give the proof of Theorem \ref{OrI_noetherian}.

\begin{proof}[Proof of Theorem \ref{OrI_noetherian}]
Suppose that there exists a rising chain of $\OrI(R)$-submodules
\[M^1\subsetneq M^2 \subsetneq \dots \subset M\]
for some $\OrI(R)$-module $M$.
Restricting $M$ and its submodules $M^i$ to $\OrIsq$, we get a chain of $\OrIsq(R)$-submodules
\[M^1_{\rm sq}\subsetneq M^2_{\rm sq} \subsetneq \dots \subset M_{\rm sq}.\]
Theorem \ref{OrIsq_noetherian} implies that this chain must stabilize at some finite $N_{\rm sq}$. Similarly, restricting to each $\OrI_I$, we get a chain of $\OrI_I(R)$-submodules
\[M^1_I\subsetneq M^2_I \subsetneq \dots \subset M_I.\]
Theorem \ref{OrInsq_noetherian} implies that each such chain must stabilize at some finite $N_I$. Therefore, the original chain must also stabilize at a finite point, namely $1 + \max(N_{\rm sq}, N_I)$.
\end{proof}

\section{Asymptotic Structure Theorem}

\label{asymp}

\subsection{Preliminaries}

Let $R$ be a finite commutative ring where 2 is a unit. Recall that a \emph{complemented category} is a symmetric monoidal category $(\Cl, \otimes)$ satisfying that:
\begin{itemize}
    \item The identity object in $\Cl$ is initial (hence there are \emph{canonical morphisms} $V\to V\otimes V'$ and $V'\to V\otimes V'$);
    \item Every morphism in $\Cl$ is a monomorphism;
    \item The map $\Hom_{\Cl}(V\otimes V', W) \to \Hom_{\Cl}(V,W) \times \Hom_{\Cl}(V',W)$, defined by composing with the canonical morphisms, is injective;
    \item For every subobject $V$ of $W$, there exists a unique subobject $V'$ of $W$ such that there is an isomorphism $V\otimes V' \to W$, which satisfies that the compositions $V \to V\otimes V' \to W$ and $V' \to V\otimes V' \to W$ are the inclusions.
    \item The map $S_n \wr \Aut_{\Cl}(V) \to \Aut_{\Cl}(V^n)$ is injective.
\end{itemize}
Furthermore, a complemented category $\Cl$ is \emph{generated by} an object $X$ if any object is isomorphic to $X^n$ for an unique $n\ge 0$. We recall the following theorem, whose three-part structure reflects a parallel with the multiplicity stability theorem described in $\mathsection 1$: 

\begin{thm}[Theorem F, \cite{Putman_2017}]
\label{asy0}
Let $(\Cl,\otimes)$ be a complemented category with a generator $X$; assume that the category of $\Cl$-modules is locally Noetherian, and let $M$ be a finitely generated $\Cl$-module. For $N\ge 0$, denote by $\Cl^N$ the full subcategory of $\Cl$ generated by all objects of $X$-rank at most $N$. Then:
\begin{itemize}
    \item (Injective stability) If $f:V\to W$ is a morphism in $\Cl$, then $M(f): M(V)\to M(W)$ is injective when the $X$-rank of $V$ is sufficiently large.
    \item (Surjective stability) If $f:V\to W$ is a morphism in $\Cl$, then the orbit under $\Aut_{\Cl}(W)$ of the image of $M(f)$ spans $M(W)$ when the $X$-rank of $V$ is sufficiently large.
    \item (Central stability) For $N$ sufficiently large, the functor $M$ is the left Kan extension to $\Cl$ of the restriction of $M$ to $\Cl^N$.
\end{itemize}
\end{thm}

Because the category $\OrIsq(R)$ is a complemented category generated by the 1-dimensional $R$-module $(R, 1)$, and the category of $\OrIsq(R)$-modules is locally Noetherian by Theorem \ref{OrIsq_noetherian}, Theorem \ref{asy0} holds for $\OrIsq(R)$. In this section, we will prove an analogous theorem for $\OrI(R)$ (Theorems \ref{surj} and \ref{central} and Corollary \ref{inj}). 

\subsection{Surjective stability}

For simplicity, we will use $\A$ in place for $\OrI(R)$ in this section.

\begin{lem}
The category $\A = \OrI(R)$ is a complemented category.
\end{lem}

\begin{proof}
The monoidal structure is given by the orthogonal direct sum, and for a free orthogonal submodule $V\subseteq W$ the complement of $V$ is given by $V^{\perp} = \{w\in W \mid B(v,w) = 0\}$, where $B$ is the nondegenerate symmetric form equipped on $W$.
\end{proof}

We remark that $\A$ is generated by not one generator, but instead $2^q$, one for each subset $I\subseteq \{1,2,\dots,q\}$: we denote by $\XX_I$ the 1-dimensional free $R$ module equipped with the bilinear form $x = (x_1,\dots,x_q)$ where $x_i$ is a fixed nonsquare in the residue field of $R_i$ if $i\in I$, and $x_i = 1$ otherwise. The fact that they generate $\A$ follows from Proposition \ref{finite_ring_decomp}. We also remark that $\XX_I^2$ is isomorphic to $\XX_{\varnothing}^2$, because in the residue fields of $R_i$, the product of two nonsquares is a square.

\begin{lem}
\label{transitive}
For $V,W\in \A$, $\Aut_{\A}(W)$ acts transitively on $\Hom_{\A}(V,W)$. 
\end{lem}

\begin{proof}
Suppose we are given morphisms $f,g \in \Hom_{\A}(V,W)$. Suppose the complements of $f(V)$ and $g(V)$ in $W$ are respectively $U, U'$. The orthogonal modules $f(V),g(V)$ are both isomorphic to $V$, and the orthogonal modules $U, U'$ are isomorphic as well. So we get the automorphism
\[\phi: W \cong f(V)\oplus U \xrightarrow{\cong} g(V)\oplus U' \cong W\]
satisfies $\phi f = g$, as desired.
\end{proof}

\begin{cor}
\label{morphism_factor}
Let $f:U\to V$ and $g: U\to W$ be morphisms in $\A$, such that $\rk V < \rk W$. Then there exist a morphism $h: V\to W$ such that $h\circ f = g$.
\end{cor}

\begin{proof}
Because $\rk V < \rk W$, we claim there exists a morphism $i:V\to W$. To see this, suppose $B,B'$ are the bilinear forms attached to $V$ and $W$, and $\det B$ is nonsquare in $R_i$ for $i\in I_V$ while $\det B'$ is nonsquare in $R_i$ for $i\in I_W$. Then $W\cong V\oplus (R,\XX_{I_V\triangle I_W})$ ($\triangle$ means set XOR). 

By Lemma \ref{transitive}, there exists an automorphism $k:W\to W$ such that $k\circ g = i\circ f$, hence $g = (k^{-1}\circ i)\circ f$.
\end{proof}

\begin{thm}[Surjective stability]
\label{surj}
Let $M$ be a finitely generated $\A$-module. Then surjective stability holds: for an $\A$-morphism $f: V\to W$ with $\rk V$ large enough, the image of $M(f)$ spans $M(W)$ under $M(\Aut_{\A}(W))$.
\end{thm}

\begin{proof}
Suppose $f: V\to W$ is any morphism in $\A$. The span of $M(f)$ in $M(W)$, under the action of $M(\Aut_{\A}(W))$, is the span of $M(\phi \circ f)$ as $\phi$ ranges in $\Aut_{\A}(W)$. By Lemma \ref{transitive}, this is the same as the span of $M(h)$ as $h$ ranges in $\Hom_{\A}(V,W)$.

Now, because $M$ is finitely generated, there exists a generating set $\{x_i
\in M(V_i)\}$. Let $r = \max_i \rk V_i$, and consider any object $V$ with rank $\rk V \ge r+1$. Fix maps $\phi_i: V_i \to V$. For any $x\in W$, since $M$ is generated by the $x_i$'s, there exist maps $f_i: V_i \to W$ such that $x$ is in the span of the images of $M(f_i): M(V_i) \to M(W)$. 

If $\rk V < \rk W$, then by Corollary \ref{morphism_factor} there exist maps $h_i: V\to W$ such that $h_i\circ \phi_i = f_i$. This implies that $x$ lies in the span of $M(h_i)$, as desired. If $\rk V = \rk W$, then since there exists a map $f:V\to W$, $V$ must be isomorphic to $W$, so the proof of Corollary \ref{morphism_factor} still applies, and $x$ lies in the span of $M(h_i)$.
\end{proof}

\subsection{Injective and central stability} 

We will now prove injective and central stability.

\begin{Def}
Let $M$ be an $\A$-module. The \emph{torsion submodule} $M_T$ of $M$ is defined by
\[M_T(V) = \{x\in M(V) \mid \exists f: V\to W,\ M(f)(x) = 0\}.\]
\end{Def}

\begin{lem}
The torsion submodule $M_T$ is an $\A$-submodule of $M$.
\end{lem}

\begin{proof}
It suffices to show that if $x,y\in M_T(V)$, then $x+y\in M_T(V)$. Suppose $f: V\to W$ and $g: V\to W'$ such that $M(f)(x) = M(g)(y) = 0$. Consider the maps $\imath: W\hookrightarrow W\oplus W'$ and $\imath': W'\hookrightarrow W\oplus W'$. By Lemma \ref{transitive}, there exists $h\in\Aut_{\A}(W\oplus W')$ such that $h\imath' g = \imath f$. This composition maps both $x$ and $y$ to 0, hence also $x+y$.
\end{proof}

\begin{lem}
\label{torsion_vanish}
Let $M$ be finitely generated, then for all $V\in\A$ with $\rk V \gg 0$, $M_T(V) = 0$.
\end{lem}

\begin{proof}
Because the category of $\A$-modules is locally Noetherian and $M_T$ is a submodule of $M$, $M_T$ is finitely generated as well. Then all maps into sufficiently large-rank spaces are zero maps by Corollary \ref{morphism_factor}.
\end{proof}

\begin{cor}[Injective stability]
\label{inj}
Let $M$ be a finitely generated $\A$-module. Then injective stability holds: for an $\A$-morphism $f: V\to W$ with $\rk V$ large enough, $M(f)$ is injective.
\end{cor}

\begin{proof}
This is a direct consequence of the above lemma.
\end{proof}

\begin{Def}
Let $n$ be a positive integer, and let $M$ be an $\A$-module. Suppose $I$ is a subset of $\{1,2,\dots,q\}$. Define $\Sigma_{I,n} M$ to be an $\A$-module mapping each $V\in\A$ to
\[\Sigma_{I,n}M(V) = \bigoplus_{h\in \Hom_{\A}(\XX_{I}^n, V)} M(C_h),\]
where $C_h$ denotes the complement of $h(\XX_{I}^n)$ in $V$.
\end{Def}

\begin{lem}
If $M$ is finitely generated, then the $\A$-modules $\Sigma_{I,n} M$ are all finitely generated.
\end{lem}

\begin{proof}
Because $M$ is finitely generated, there exists a set of generators $x_1,\dots,x_m$, which respectively belong to $M(V_1),\dots,M(V_m)$. Notice that each $M(V_i)$ is a direct summand in $\Sigma_{I,n}M(V_i \oplus \XX_I^n)$. We claim that the images of these elements $x_1,\dots,x_m$ also generate $\Sigma_{I,n}M$. For any nonzero $\Sigma_{I,n}M(V)$, $V$ must be isomorphic to $\XX_I^n\oplus Y$ for some orthogonal module $Y$. Clearly, the direct summand $M(Y)$ inside $\Sigma_{I,n}M(\XX_I^n\oplus Y)$ is generated by the claimed set of generators; by Lemma \ref{transitive}, the other direct summands are all generated by them as well.
\end{proof}

We can endow $\Sigma_{I, \bullet} M$ with a chain complex structure as follows. Consider an object $V\in \A$, and we define the maps $d_1,\dots,d_n: \Sigma_{I,n}M(V)\to \Sigma_{I,n-1}M(V)$ as induced by the maps $\XX_{I}^{n-1} \to \XX_{I}^n$ by ``adding'' the $i$th coordinate ($1\le i\le n$). Finally, define $d = d_1-d_2+\dots+(-1)^{n-1}d_n$. It is straightforward to check that $\Sigma_{\bullet} M$ is a chain complex. 

\begin{thm}
Let $M$ be finitely generated. Fix $n$ a positive integer, and $I\subseteq \{1,2,\dots,q\}$. Then the chain complex
\[\Sigma_{I,n}M(V) \to \dots \to \Sigma_{I,1} M(V) \to M(V) \to 0\]
is exact for all $V$ with sufficiently large rank (as free $R$-modules).
\end{thm}

\begin{proof}
It suffices to prove that $(\HH_i(\Sigma_{I, \bullet}M))(V) = (\HH_i(\Sigma_{I, \bullet}M)_T)(V)$, since then it follows from Lemma \ref{torsion_vanish} that $(\HH_i(\Sigma_{\bullet}M))(V) = 0$ for sufficiently large $\rk V$, which implies the exactness of the chain complex. To prove the claimed fact, we only need to show that the map $(\HH_i(\Sigma_{I, \bullet}M))(V) \to (\HH_i(\Sigma_{I, \bullet}M))(V\oplus \XX_{I})$, induced by the canonical morphism $V\to V\oplus \XX_I$, is the zero map. But then the argument in Lemma 3.11 of \cite{Putman_2017} applies verbatim.
\end{proof}

Now, we are ready to show central stability. The argument closely follows the proof of Theorem F in \cite{Putman_2017}. The main modification is the use of the chain complex $\Sigma_{I,n} M$, which is parametrized by $I$ to handle the various isomorphism classes of orthogonal modules for any given rank.

\begin{thm}[Central stability]
\label{central}
Let $\A_{\le N}$ denote the full subcategory spanned by the following objects:
\begin{itemize}
    \item objects with $R$-rank at most $N-1$;
    \item objects isometric to the rank-$N$ orthogonal $R$-module equipped with the bilinear form whose matrix representation under the standard basis is the identity.
\end{itemize} 
Let $M$ be a finitely generated $\A$-module. Then for all $N$ sufficiently large, $M$ is the left Kan extension to $M\mid_{\A_{\le N}}$ along the inclusion functor $p: \A_{\le N} \to \A$.
\end{thm}

\begin{proof}
Let $M'$ be the desired left Kan extension, then the universal property gives a natural transformation $\phi: M'\to M$ such that $\phi_V: M'(V)\to M(V)$ are isomorphisms for all $V$ with $\rk V\le N-1$. (Here we choose $N$ sufficiently large so that $\Sigma_{I,2}M \to \Sigma_{I,1}M \to M \to 0$ is exact whenever $\rk V \ge N$.) It suffices to prove that $\phi_V$ are isomorphisms for all $V$. Unlike in the proof of Theorem F in \cite{Putman_2017}, the chain complex $\Sigma_{I,n}M$ we choose will depend on the isomorphism class of $V$, which is the key difference from that proof.

Induct on $\rk V$, and we assume $\phi_V$ is an isomorphism for all $\rk V\le r-1$ ($r\ge N+1$). Fix $V$ to be a rank-$r$ object, and without loss of generality $V = \XX_{\varnothing}^{r-1} \oplus \XX_{I_0}$. The natural transformation $\phi$ induces natural transformations $\Sigma_{I,i}M \to \Sigma_{I,i}M'$ for each $I$, and by definition we know the induced maps $\Sigma_{I,i}M(V) \to \Sigma_{I,i}M'(V)$ are isomorphisms for each $i\ge 1$ and $I\subseteq \{1,\dots,q\}$. 

Consider the commutative diagram
\[
\begin{tikzcd}
\Sigma_{I_0,2}M'(V) \arrow[r] \arrow[d, "\cong"] & \Sigma_{I_0,1}M'(V) \arrow[r] \arrow[d, "\cong"] & M'(V) \arrow[r] \arrow[d, "\phi_V"] & 0 \\
\Sigma_{I_0,2}M(V) \arrow[r] & \Sigma_{I_0,1}M(V) \arrow[r] & M(V) \arrow[r] & 0, \tag{1}
\end{tikzcd}
\]
whose bottom row is exact. We now wish to prove that the map $\Sigma_{I_0, 1}M'(V) \to M'(V)$ is surjective, from which it will follow that $\phi_V$ is an isomorphism by a simple diagram chase on (1).

By definition of the Kan extension, for $V\in\A$, $M'(V)$ is the colimit
\[\varinjlim ((p\downarrow V) \to \A_{\le N} \xrightarrow{M} \mathbf{Mod}_{\KK}).\]
Furthermore, it is easy to see that is a filtered colimit (due to our definition of $\A_{\le N}$), so as a set it is explicitly given by
\[M'(V) = \left(\coprod_{\substack{(U,f): U\xrightarrow{f}V \\ U\in \A_{\le N}}} M(U)\right) \Bigg/ \sim\]
where $\sim$ is the usual equivalence relation.
By Corollary \ref{morphism_factor}, any map $f:U\to V$ from an object $U$ of rank at most $N$ to $V$ factors through an object $W$ of rank $r-1$. Furthermore, we can suppose $W = C_h$, where $h:\XX_{I_0} \to V$ and $C_h$ is the complement of the image of $h$ in $V$. The inclusion $W\to V$ induces a map $u: M'(W) \to M'(V)$. Define another map $u': M'(W) \to M'(V)$ by the universal property of $M'(W)$ as a filtered colimit. It suffices to show that $u$ and $u'$ are the same map, since then it would imply the surjectivity of $\Sigma_{I_0, 1}M'(V) \to M'(V)$. 

To do this, consider the objects $\{M'(U): U\in \A_{\le N}, f: U\to V\}$. These, along with the morphisms $M'(U)\to M'(U')$ defined by
\[M'(U)\xrightarrow{\cong} M(U) \to M(U')\xrightarrow{\cong} M'(U'),\]
form another diagram of shape $(p\downarrow W)$ in $\A$. The colimit of this is again $M'(W)$, and the map $u: M'(W)\to M'(V)$ is precisely the (unique) map provided by the universal property. Therefore, we conclude that $u = u'$, which finishes the proof.
\end{proof}

\section{Twisted Homological Stability}

\label{twist}

In this section, we fix $R$ to be a finite field with $\operatorname{char} R > 2$. In this case, there are 2 isomorphism classes of non-degenerate symmetric bilinear forms in each dimension.

Define, for an object $(V,B)\in\OrI(R)$, $O(V,B)$ to be the orthogonal group associated with $(V,B)$ (i.e. the group of $R$-linear isomorphisms $V\to V$ preserving the bilinear form $B$). Then for any morphism $(V,B) \to (W,B')$, there is an induced map $O(V,B)\to O(W,B')$ given by mapping $f$ to $f\oplus \mathrm{id}$, where $\mathrm{id}$ is the identity map on the complement of $(V,B)$ in $(W,B')$.

In this section, we prove the following twisted homological stability theorem:

\begin{thm}
\label{twisted}
Let $R$ be a finite field with $\operatorname{char} R > 2$. Let $M$ be a finitely generated $\OrI(R)$-module over $\KK = \Z/\ell\Z$, where $\ell$ is a prime not equal to $\operatorname{char} R$. Fix any index $k\ge 0$. Consider a morphism $(V,B)\to(W,B')$ in $\OrI(R)$. As explained above, this induces a map on the homologies 
\[H_k(O(V,B);M(V,B)) \to H_k(O(W,B');M(W,B')).\] 
Then this map is an isomorphism for all $V$ with sufficiently large rank $n$.
\end{thm}

First, we need to state several lemmas.

\begin{lem}
\label{aut}
Let $V,V'\in \OrI(R)$, and let $i: V\to V\oplus V'$ be the canonical injection given by $v\mapsto v\oplus 0$. Then
\[\{\tau \in \Aut(V\oplus V'): \tau i = i \} \cong \Aut(V').\]
\end{lem}

\begin{proof}
See the proof of \cite{Putman_2017}, Lemma 3.2. Because $\OrI(R)$ is a complemented category, the conclusion there applies to $\OrI(R)$.
\end{proof}

\begin{lem}
\label{induced}
Let $V, W\in\OrI(R)$, then, as $\Aut(V)$-representations,
\[\KK[\Hom(V, V\oplus W)] \cong \KK[\Aut(V\oplus W)/\Aut(W)] \cong \Ind_{\Aut(W)}^{\Aut(V\oplus W)} (\KK).\]
\end{lem}

\begin{proof}
By Lemma $\ref{transitive}$, $\Aut(V\oplus W)$ acts transitively on $\Hom(V,V\oplus W)$. By Lemma $\ref{aut}$, the stabilizer of any element $i\in \Hom(V,V\oplus W)$ is isomorphic to $\Aut(W)$, which proves the first equality. The second equality follows by definition of an induced representation.
\end{proof}

\begin{cor} 
\label{shapirocor}
For any $V,W\in\OrI(R)$,
\[H_k(\Aut(V\oplus W); \KK[\Hom(V, V\oplus W)]) \cong H_k(\Aut(W); \KK).\]
\end{cor}

\begin{proof}
We have
\begin{align*}
H_k(\Aut(W); \KK) &\cong H_k(\Aut(V\oplus W); \Ind_{\Aut(W)}^{\Aut(V\oplus W)}(\KK)) \\ &\cong H_k(\Aut(V\oplus W); \KK[\Hom(V,V\oplus W)]).
\end{align*}
where the first step uses Shapiro's lemma for group cohomology ($\Aut(W)\le \Aut(V\oplus W)$ by acting trivially on $V$), and the second step uses Lemma \ref{induced}).
\end{proof}

Now, we prove the main theorem of this section.  

\begin{proof}[Proof of Theorem \ref{twisted}]
Let $X,Y$ respectively denote the square and nonsquare generator of $\OrI(R)$; in the notation of $\mathsection$\ref{asymp}, they would be $X = \XX_{\varnothing}$ and $Y = \XX_{\{1\}}$. Also, fix an isomorphism $X^2\cong Y^2$

First, we consider the case where we impose the following two extra conditions on $M$:
\begin{enumerate}[(i)]
    \item For any morphism $f: V\to W$ in $\OrI(R)$, the map $M(f): M(V)\to M(W)$ is injective;
    \item There exists a set of generators of $M$ such that each generator lies in $M(V)$, where $V \cong X^r \oplus Y$ for some $r\ge 0$ (possibly depending on the generator).
\end{enumerate}

In this case, we will show that for any fixed $k$, the map 
\[H_k(\Aut(X^n\oplus Y); M(X^n\oplus Y)) \to H_k(\Aut(X^{n+1}\oplus Y); M(X^{n+1}\oplus Y)) \tag{2}\]
is an isomorphism for all $n\gg 0$.

Because $M$ is finitely generated, by Lemma \ref{fingen_quotient}, $M$ is a quotient module of a projective module $P_0$, which is the direct sum of finitely many representable $\OrI(R)$-modules. Furthermore, because of condition (ii), we can choose $P_0$ so that it is the direct sum of representable functors based at objects isomorphic to those of the form $X^r\oplus Y$. The kernel of $P_0 \to M$ is finitely generated by local Noetherianity. Consider the following commutative diagram, where $f:V\to W$ is a morphism in $\OrI(R)$:
\[
\begin{tikzcd}
P_0(V) \arrow[r, "P_0(f)"] \arrow[d] & P_0(W) \arrow[d] \\
M(V) \arrow[r, "M(f)"] & M(W).
\end{tikzcd}
\]
Because $M(f)$ is injective, we see that if $v\in P_0(V)$ and $P_0(f)(v) \in \ker(P_0(W)\to M(W))$, then $v\in \ker(P_0(V)\to M(V))$. From this, we deduce that the kernel $\ker(P_0\to M)$ must also satisfy its analogous properties (i) and (ii).

Therefore, we can repeat this process to extend this to a projective resolution $\bar{C}$ of $M$ by direct sums of finitely many representable $\OrI(R)$-modules based at objects isomorphic to $X^r\oplus Y$:
\[\bar{C}: \dots \to P_2 \to P_1 \to P_0 \to M \to 0.\]

Now, if we delete $M$ from the chain complex $\bar{C}$ to produce
\[C: \dots \to P_2 \to P_1 \to P_0 \to 0,\]
then for $V$ an object in $\OrI(R)$,
\[H_{k}(\Aut(V); M(V)) \cong H_{k}(\Aut(V); C),\]
where the right hand side is group homology with coefficients in a chain complex.
There exists a spectral sequence (cf. \cite{browncohomology}):
\[E_{pq}^1 = H_p(\Aut(V); P_q(V)) \Longrightarrow H_{p+q}(\Aut(V); C). \tag{3}\]

Because each $P_q$ is the direct sum of representable modules based at objects isomorphic to some $X^r \oplus Y$, we see (using Corollary \ref{shapirocor}) that for any fixed $p,q,n$, there is a map
\begin{align*}
\bigoplus_i H_p(\Aut(X^i); \KK) &\cong H_p(\Aut(X^n\oplus Y); P_q(X^n\oplus Y)) \\ &\to H_p(\Aut(X^{n+1}\oplus Y); P_q(X^{n+1}\oplus Y)) \cong \bigoplus_i H_p(\Aut(X^{1+i}); \KK).
\end{align*}

From \cite{Friedlander1976}, each map $H_p(\Aut(X^i);\KK) \to H_p(\Aut(X^{1+i});\KK)$
is an isomorphism for all $i \gg 0$. Hence, we conclude that
\[H_p(\Aut(X^n\oplus Y); P_q(X^n\oplus Y)) \xrightarrow{\cong} H_p(\Aut(X^{n+1}\oplus Y); P_q(X^{n+1}\oplus Y))\]
for all $n \gg 0$. Applying the spectral sequence (3), we conclude that for any fixed $k = p+q$, the map (2) is indeed an isomorphism for all $n\gg 0$.

Next, we consider the second case where $M$ satisfies (i) and the following condition (ii'), instead of (ii):
\begin{itemize}
    \item[(ii')] There exists a set of generators of $M$ such that each generator lies in $M(V)$, where $V \cong X^r$ for some $r\ge 0$ (possibly depending on the generator).
\end{itemize}

Using the exact same argument, we can similarly show that for any fixed $k$, the map 
\[H_k(\Aut(X^n); M(X^n)) \to H_k(\Aut(X^{n+1}); M(X^{n+1})) \tag{4}\]
is an isomorphism for all $n\gg 0$. 

We now consider the general case where $M$ is any finitely generated $\OrI(R)$-module. Fix a large enough positive integer $N$ according to Theorem \ref{surj} and Corollary \ref{inj}, such that injective and surjective stability holds for all objects with rank at least $N$. Let $M_1$ be the submodule of $M$ given by 
\[M_1(V) =
\begin{cases}
M(V) & \text{if either } \rk V > N \text{ or } V\cong X^{N-1}\oplus Y; \\
0 & \text{otherwise,}
\end{cases}
\]
and for any morphism $f: V\to W$, $M_1(f)$ is the zero map when $M_1(V) = 0$, and is the same as $M(f)$ otherwise. Then by local Noetherianity, $M_1$ is a finitely generated module satisfying conditions (i) and (ii), and $M_1(V)\cong M(V)$ for all $V$ of large enough rank. As a result,
\begin{align*}
    H_k(\Aut(X^n \oplus Y); M(X^n \oplus Y)) &= H_k(\Aut(X^n \oplus Y); M_1(X^n \oplus Y)) \\
    &\cong H_k(\Aut(X^{n+1} \oplus Y); M_1(X^{n+1} \oplus Y)) \\
    &= H_k(\Aut(X^{n+1} \oplus Y); M(X^{n+1} \oplus Y))
\end{align*}
is an isomorphism for all $n\gg 0$.

Similarly, let $M_2$ be the submodule of $M$ given by 
\[M_2(V) =
\begin{cases}
M(V) & \text{if either } \rk V > N \text{ or } V\cong X^{N}; \\
0 & \text{otherwise,}
\end{cases}
\]
and for any morphism $f: V\to W$, $M_2(f)$ is the zero map when $M_2(V) = 0$, and is the same as $M(f)$ otherwise. Then by local Noetherianity, $M_2$ is a finitely generated module satisfying conditions (i) and (ii'), and $M_2(V)\cong M(V)$ for all $V$ of large enough rank. As a result,
\begin{align*}
    H_k(\Aut(X^n); M(X^n)) &= H_k(\Aut(X^n); M_2(X^n)) \\
    &\cong H_k(\Aut(X^{n+1}); M_2(X^{n+1})) \\
    &= H_k(\Aut(X^{n+1}); M(X^{n+1}))
\end{align*}
is an isomorphism for all $n\gg 0$.

Finally, because the composition of isometries $X^{n-2}\oplus Y \to X^n \to X^{n}\oplus Y \to X^{n+2}$, given by the inclusions and isomorphisms
\[X^{n-2}\oplus Y \to X^{n-2}\oplus Y^2 \cong X^n \to X^{n}\oplus Y \to X^n \oplus Y^2 \cong X^{n+2}\]
induces the maps
\[\Aut(X^{n-2}\oplus Y) \to \Aut(X^{n})\to \Aut(X^{n}\oplus Y) \to \Aut(X^{n+2}),\]
which induces the maps
\begin{align*}
H_k(\Aut(X^{n-2}\oplus Y); M(X^{n-2}\oplus Y)) &\xrightarrow{f} H_k(\Aut(X^n); M(X^n)) \\ &\xrightarrow{g} H_k(\Aut(X^{n}\oplus Y); M(X^{n}\oplus Y)) \\ & \xrightarrow{h} H_k(\Aut(X^{n+2}); M(X^{n+2})),
\end{align*}
and since $gf, hg$ are both isomorphisms, we conclude that $g$ is an isomorphism as well. Because any object $(V,B) \in \OrI(R)$ is isomorphic to either $X^n$ or $X^n\oplus Y$ ($n\ge 0$), these are enough to imply Theorem \ref{twisted}.
\end{proof}

As a corollary, we show the following homological stability result with untwisted coefficients:

\begin{cor}
Under assumptions stated in Theorem \ref{twisted}, the map \[H_k(O(V,B);\KK) \to H_k(O(W,B');\KK)\] is an isomorphism for all $V$ with sufficiently large rank $n$.
\end{cor}

\begin{proof}
Take $m$ sufficiently large. Consider the finitely generated module $P_{\OrI(R), X^m}$. Then for any object $V\in\OrI(R)$,
\begin{align*}
    H_k(\Aut(V); \KK) &= H_k(\Aut(V\oplus X^m); \KK[\Hom(X^m, V\oplus X^m)]) \\
    &\cong H_k(\Aut(A\oplus X^{m+1}); \KK[\Hom(X^m, V\oplus X^{m+1})]) \\
    &= H_k(\Aut(V\oplus X); \KK),
\end{align*}
where the first and third steps follow from \ref{shapirocor}, and the second step follows from Theorem \ref{twisted} (recall we took $m$ sufficiently large so that this is an isomorphism). Similarly,
\begin{align*}
    H_k(\Aut(V); \KK) &= H_k(\Aut(V\oplus X^m); \KK[\Hom(X^m, V\oplus X^m)]) \\
    &\cong H_k(\Aut(V\oplus X^{m} \oplus Y); \KK[\Hom(X^m, V\oplus X^{m} \oplus Y)]) \\
    &= H_k(\Aut(V\oplus Y); \KK).
\end{align*}
These are enough to imply the conclusion.
\end{proof}

We remark that this partially generalizes the main theorem in \cite{charney1987} under our assumptions for $R$ and $\KK$.

\printbibliography

\end{document}